\documentclass[11pt]{amsart}

\newcommand{\area}{{\rm Area}}
\newcommand{\length}{{\rm Length}}

\newcommand{\D}{{\mathbf D}}
\newcommand{\id}{{\rm id}}
\usepackage{amsthm,amsmath,amsfonts, amssymb}
\usepackage[all]{xy}
\usepackage{psfrag}
\usepackage[dvips]{graphicx}

\usepackage[latin1]{inputenc}

\makeatletter
\def\th@alexnormal{%
\let\thm@indent\noindent % no indent
\thm@headfont{\bfseries}% heading font is bold
\normalfont
}
\makeatother

\makeatletter
\def\th@alexit{%
\let\thm@indent\noindent % no indent
\thm@headfont{\bfseries}% heading font is bold
\normalfont
\fontshape{it}
\selectfont
}
\makeatother

\theoremstyle{alexit}

\newtheorem{theorem}[equation]{Theorem}
\newtheorem{proposition}[equation]{Proposition}
\newtheorem{lemma}[equation]{Lemma}

\newtheorem{corollary}[equation]{Corollary}

\newtheorem{question}[equation]{Question}

\theoremstyle{remark}
\newtheorem{remark}[equation]{Remark}
\theoremstyle{definition}

\numberwithin{equation}{subsection}

\begin{document}
\author{Alexandre Girouard}
\address{School of Mathematics, Cardiff University,
Senghennydd Road, Cardiff, Wales,  CF24 4AG, UK}
\email{GirouardA@cardiff.ac.uk}

\author{Iosif Polterovich}\thanks{The second author  is  supported by NSERC and FQRNT}
\address{D\'epartement de math\'ematiques et de
statistique, Universit\'e de Montr\'eal, C. P. 6128,
Succ. Centre-ville, Montr\'eal, Qu\'ebec,  H3C 3J7,  Canada}
\email{iossif@dms.umontreal.ca}

\title[Shape optimization for Neumann and Steklov eigenvalues ]
{Shape optimization for low Neumann and Steklov eigenvalues}

\begin{abstract}
We give an overview of  results on shape optimization
for low eigenvalues of the Laplacian on bounded planar domains 
with Neumann and Steklov boundary conditions.  These
results share a common feature: they are proved using  methods of complex
analysis.  In particular, we present  modernized  proofs of  the
classical inequalities due to  Szeg\"o  and Weinstock  for the first
nonzero Neumann and Steklov eigenvalues. We also extend the inequality for  the second nonzero Neumann eigenvalue,  obtained recently by
Nadirashvili and the authors,   to non-homogeneous membranes with
log-subharmonic densities.  In the homogeneous case,  we show that
this inequality is strict, which implies  that the maximum of the
second nonzero Neumann eigenvalue is not attained in the class of
simply-connected membranes of  a given mass. The same is true for the second nonzero Steklov eigenvalue, as follows from our results on  the Hersch--Payne--Schiffer inequalities.  

\end{abstract}
\maketitle
%\tableofcontents

\section{Introduction and main results}
\subsection{Neumann and Steklov eigenvalue problems}
Let $\Omega$ be a simply-connected bounded planar domain with
Lipschitz boundary. Consider the {\it Neumann} and {\it Steklov}
eigenvalue problems on $\Omega$:
\begin{equation}
\label{Neumann}
-\Delta u =\mu u \mbox{ in } \Omega \,\,\,{\rm and} \,\,\,
\frac{\partial u}{\partial n}=0 \mbox{ on } {\partial \Omega},
\end{equation}
\begin{equation}
\label{Steklov}
\Delta u = 0 \mbox{ in } \Omega \,\,\, {\rm  and} \,\,\, \frac{\partial
  u}{\partial n} =\sigma u \mbox{ on } {\partial \Omega}. 
\end{equation}
Here
$\displaystyle\Delta=\partial_x^2+\partial_y^2$
is the  Laplace operator and $\frac{\partial}{\partial n}$ is the
outward normal derivative. Both problems have discrete spectra
$$0=\mu_0<\mu_1(\Omega)\leq \mu_2(\Omega)\leq
\mu_3(\Omega)\leq\cdots\nearrow\infty,$$ 
$$0=\sigma_0<\sigma_1(\Omega)\leq \sigma_2(\Omega)\leq
\sigma_3(\Omega)\leq\cdots\nearrow\infty,$$
starting  with the simple eigenvalues $\mu_0=0$ and $\sigma_0=0$, which
correspond to constant eigenfunctions. The eigenvalues $\mu_k$ and
$\sigma_k$ satisfy the following variational characterizations:
\begin{gather}\label{RayleighNeumannOriginal}
\mu_k(\Omega)=\inf_{U_k} \sup_{0\neq u\in U_k}
  \frac{\int_{\Omega}|\nabla u|^2\,dz}
  {\int_{\Omega} u^2\,dz}, \,\,\,\,\, k=1,2,\dots
\end{gather}
\begin{gather}\label{RayleighSteklovOriginal}
\sigma_k(\Omega)=\inf_{E_k} \sup_{0\neq u\in E_k}
  \frac{\int_{\Omega}|\nabla u|^2\,dz}
  {\int_{\partial\Omega} u^2\,ds}, \,\,\,\,\, k=1,2,\dots
\end{gather}
The infima are taken over all $k$--dimensional subspaces $U_k$ and $E_k$ of
the Sobolev space $H^1(\Omega)$ which are  orthogonal to
constants on $\Omega$ and $\partial \Omega$,  respectively. 
\begin{remark}
Here and further on we identify ${\bf R}^2$ with the complex plane~$\bf C$
and set $z=(x,y)$. We  write
$dz=dx\,dy$ for the Lebesgue measure.
\end{remark}

\subsection{Shape optimization}\label{shape}
Both Neumann and Steklov eigenvalue problems describe
the vibration of a free membrane. In the Neumann case the membrane
is homogeneous, while in the Steklov case the whole mass of the
membrane is uniformly distributed on $\partial \Omega$.
Therefore, we may define the {\it mass} of the membrane $\Omega$ by setting
\begin{equation}
\label{masss}
M(\Omega)=\begin{cases} \area(\Omega) &{\rm in \,\,\, the \,\,\,
Neumann \,\,\, case},\cr \length(\partial \Omega) &{\rm in \,\,\,
the\,\,\, Steklov \,\,\, case.}\cr
\end{cases}
\end{equation}
In this survey we focus on the following shape optimization problem.
\begin{question}
\label{quest}
How large can $\mu_k$ and $\sigma_k$ be on a me\-m\-brane of a given
mass?
\end{question}
In 1954,  this problem was solved by G.~Szeg\"o for $\mu_1$ and by
R.~Weinstock for $\sigma_1$. Let 
$\mathbf{D}=\{ z \in \mathbf{C}\,|\, |z| <1 \}$ be the open unit disk.
\begin{theorem}\label{TheoremSzego} {\bf ([Sz])}
  Let $\Omega$ be a simply-connected bounded planar domain with Lipschitz boundary. Then
  \begin{gather}\label{ineqSzego}
    \mu_1(\Omega) M(\Omega) \leq \mu_1(\mathbf{D})\pi \approx 3.39 \pi,
  \end{gather}
  with equality if and only if $\Omega$ is a disk.
\end{theorem}
Szeg\"o's inequality was later generalized by H. Weinberger
\cite{Wein} to arbitrary (not necessarily simply--connected)
domains in any dimension.
\begin{theorem}\label{TheoremWeinstock} {\bf ([Weinst])}
  Let $\Omega$ be a simply-connected bounded planar domain with
  Lipschitz boundary. Then
  \begin{equation}
  \label{weinstineq}
    \sigma_1(\Omega)\, M(\Omega) \leq 2\pi,
  \end{equation}
with equality if and only if $\Omega$ is a disk.
\end{theorem}
Many results were motivated by Weinstock's inequality: see, for instance, \cite{Bandle3, Brock, Dittmar, Edward, HP, HPS1, HPS2}.

\medskip

Recently,  analogues of Theorems \ref{TheoremSzego} and \ref{TheoremWeinstock}
for the second nonzero Neumann and Steklov eigenvalues were proved in \cite{GNP} and \cite{GP}. %$\mu_2$ and $\sigma_2$. 
%The following theorem is an extension of~\cite{GNP} were
%the corresponding inequality was not proved to be strict.
\begin{theorem}\label{ourThmNeumann} {\bf (cf.  [GNP])}
  (i) Let $\Omega$ be a simply-connected bounded planar domain with
  Lipschitz boundary. Then
  \begin{equation}\label{ourineqneum}
    \mu_2(\Omega)M(\Omega)< 2\,\mu_1(\mathbf{D})\,\pi \approx
    6.78 \pi.
  \end{equation}
  \noindent (ii) There exists a family of simply-connected bounded
  Lipschitz domains $\Omega_\varepsilon  \subset \mathbf{R}^2$, 
  degenerating to the disjoint union of two
  identical disks as $\varepsilon\to 0+$,  such that
  \begin{equation*}
    \lim_{\varepsilon\rightarrow 0+}\sigma_2(\Omega_\varepsilon)
    \, M(\Omega_\varepsilon) = 2\,\mu_1(\mathbf{D})\,\pi.
  \end{equation*}
\end{theorem}
Note that  inequality \eqref{ourineqneum} is {\it strict}, and hence 
Theorem \ref{ourThmNeumann}  is a slight improvement upon \cite[Theorem 1.1.3]{GNP}.

Theorem \ref{ourThmNeumann} implies the P\'olya conjecture
\cite{Polya} for the second nonzero Neumann eigenvalue of a simply-connected
bounded planar domain:
$$\mu_2(\Omega) \area(\Omega) \le 8 \pi.$$ The best previous estimate
on $\mu_2$ was obtained in \cite{Kroger}:
$$\mu_2(\Omega) \area(\Omega)\le 16 \pi.$$
\begin{theorem}\label{ourThmSteklov}   {\bf ([GP])}
  (i) Let $\Omega$ be a simply-connected bounded planar domain with
  Lipschitz boundary. Then
  \begin{equation}
  \label{ourineqst}
    \sigma_2(\Omega) \, M(\Omega) < 4 \pi.
  \end{equation}
  \noindent
  (ii) There exists a family of simply-connected bounded
  Lipschitz domains $\Omega_\varepsilon  \subset \mathbf{R}^2$, degenerating to the disjoint union of two
  identical disks as $\varepsilon\to 0+$,  such that
  \begin{equation*}
    \lim_{\varepsilon\rightarrow 0+}\sigma_2(\Omega_\varepsilon)
    \, M(\Omega_\varepsilon) = 4\pi.
  \end{equation*}
\end{theorem}
The proofs of Theorems ~\ref{ourThmNeumann} and \ref{ourThmSteklov} use similar techniques.
Inequality \eqref{ourineqst} is a slight sharpening in the case $k=2$ of the 
estimate
\begin{equation}
\label{hps}
\sigma_k(\Omega) \, M(\Omega) \le 2 \pi k, \,\,\, k=1,2,\dots,
\end{equation}
obtained  earlier by Hersch--Payne--Schiffer \cite[p. 102]{HPS1} by a completely different method. 
Our approach allows to show that  \eqref{hps} is  strict for $k=2$, similarly to  \eqref{ourineqneum}.  Note 
that this contrasts with estimates \eqref{ineqSzego} and \eqref{weinstineq}.  In particular, we have the following
\begin{corollary}
The maximal values of the second nonzero Neumann and Steklov eigenvalues are not attained in the class of simply--connected Lipschitz domains of a given mass. 
\end{corollary}
\begin{remark} Theorems \ref{TheoremSzego} and \ref{TheoremWeinstock} are analogues of the Faber--Krahn inequality for the first Dirichlet eigenvalue (\cite{Faber, Krahn1},  
\cite[section 3.2]{Henrot}), while Theorems \ref{ourThmNeumann} and \ref{ourThmSteklov} are similar to Krahn's inequality for the second Dirichlet eigenvalue (\cite{Krahn2}, \cite[section 4.1]{Henrot}). Note that  the equalities in the estimates of  Faber--Krahn and Krahn  are also attained for  the disk  and  the disjoint union of two identical disks, respectively.
\end{remark}
%It is not known whether the disjoint union of two identical disks is
%the only optimal shape for $\mu_2$ and $\sigma_2$.
\subsection{Higher eigenvalues}
\label{higher}
One could ask whether $\mu_k$ and $\sigma_k$  are maximized in the limit by the disjoint union of 
$k$ identical disks for all $k\ge 1$.   In \cite{GP} we show that this is indeed true for all Steklov eigenvalues,  and that the Hersch--Payne--Schiffer inequality  \eqref{hps} is {\it sharp}
for all $k\ge 1$. This gives an almost  complete answer to Question \ref{quest} in the Steklov case.
It  remains to check whether \eqref{hps} is always a {\it strict} inequality. 
 %It would be interesting to show that it is also {\it strict} for all $k\ge 1$.

For Neumann eigenvalues the situation is more complicated. Indeed, if all $\mu_k$ are maximized in the limit by the disjoint union of $k$ identical disks, then for any
simply-connected domain $\Omega$ and each integer $k\geq 1$, 
$$\mu_k(\Omega)M(\Omega)\leq k\,\mu_1(\mathbf{D})\,\pi\approx 3.39 k\pi.$$
However, this is false for {\it any} domain $\Omega$, because 
$$\mu_k(\Omega)M(\Omega)\sim 4k\pi\quad \mbox{ as }k\rightarrow\infty$$
according to Weyl's law \cite[p. 31]{Chavel}. 

\subsection{Non-homogeneous membranes}
\label{shapeDensity}
\subsubsection*{Neumann problem}
Let $\rho \in L^\infty(\Omega)$ be a positive  function representing 
the density of a non-homogeneous membrane $\Omega$ of  mass
$$M(\Omega)=\int_{\Omega}\rho(z)\,dz.$$
In this context, the Neumann eigenvalue problem becomes
\begin{equation}
  -\Delta u =\mu\, \rho\, u \mbox{\,\, in } \Omega \,\,\,{\rm and} \,\,\,
  \frac{\partial u}{\partial n}=0 \mbox{ on }\partial \Omega.
\end{equation}
It also has a discrete spectrum
$$0=\mu_0<\mu_1(\Omega,\rho)\leq \mu_2(\Omega,\rho)\leq
\mu_3(\Omega,\rho)\leq\cdots\nearrow\infty.$$
The following result  shows that inequalities~\eqref{ineqSzego} and ~\eqref{ourineqneum}  can be
generalized to  this setting. 
\begin{theorem}\label{ThmNeumannDensity}
  Let $\Omega$ be a simply-connected domain with Lipschitz boundary
  and density $\rho\in C^2(\overline{\Omega})$. If $\Delta\log\rho\geq 0$, then
  \begin{equation}\label{ineqBandle}
    \mu_1(\Omega,\rho) M(\Omega) \leq \mu_1(\mathbf{D})\pi, % \,\,\, (\cite{Bandle2}) 
  \end{equation}
  \begin{equation}\label{ourineqneumDensity}
    \mu_2(\Omega,\rho)M(\Omega)\leq 2\,\mu_1(\mathbf{D})\,\pi.
  \end{equation}
\end{theorem}
Inequality \eqref{ineqBandle}  was proved by  C. Bandle~(\cite{Bandle2},  
\cite[p. 121-128]{Bandle1}).  %We present a new proof of this result in section \ref{SectionProofSzego},
%see Remark~\ref{remarkBandle}.
To the best of our knowledge, estimate \eqref{ourineqneumDensity} is new. Its proof 
is very similar to that of  ~\eqref{ourineqneum}.  We believe that 
inequality~\eqref{ourineqneumDensity} is strict, and it would be interesting to establish this fact.

\begin{remark}
Recall that the  Gaussian curvature of the Riemannian metric $g=\rho(x,y)(dx^2+dy^2)$
is given by the well known formula \cite[Theorem  13.1.3]{DFN}: %for Gaussian curvature in
%isothermal coordinates (cf.~\cite[p. 663]{BR}):
\begin{gather}\label{isothermalGauss}
  K_g=-\frac{1}{2\rho}\Delta\log\rho.
\end{gather}
It follows that the condition $\Delta\log\rho\geq 0$ is equivalent to
$K_g\leq 0$. In other words, {\it log-subharmonic} densities correspond to nonpositively curved membranes.
 \end{remark}
\begin{remark}
If we impose no restriction on the density $\rho$, then it is easy to
see that maximizing $\mu_k$ for  simply--connected
membranes is equivalent to finding a Riemannian metric on the sphere
that maximizes the $k$-th  Laplace-Beltrami eigenvalue. It follows
from ~\cite[Corollary 1]{ElCol} that
\begin{gather}
\label{elcol}
  \sup_{\Omega, \rho}\mu_k(\Omega,\rho)M(\Omega)\geq 8k\pi
  \quad\mbox{ for each }k\geq 1.
\end{gather}
For $k=1, 2$ this is an equality, as was shown in \cite{Hersch} and ~\cite{Nad}. 
Interestingly enough, extremal metrics for
the first two eigenvalues on the sphere  resemble the extremal domains
described in section \ref{shape}:   the first eigenvalue is maximized
by  the round sphere, and the supremum for the second eigenvalue is
attained in the limit by a sequence of metrics converging to the
disjoint union of two identical round spheres.
In fact, in \cite{Nad} it is conjectured  that \eqref{elcol} is an equality for each
$k \ge 1$, and that the supremum is attained in the limit by the densities corresponding to a family of surfaces degenerating
to the disjoint union of $k$ identical round spheres. If true, this would be similar to the case of higher Steklov eigenvalues, see 
section \ref{higher}.

\end{remark}
\subsubsection*{Steklov problem}
Let $\rho \in L^\infty(\partial\Omega)$  be a positive  function representing the boundary density of the membrane $\Omega$ 
of total mass
$$M(\Omega)=\int_{\partial\Omega}\rho(s)\,ds.$$
The non-homogeneous Steklov eigenvalue problem is given by
\begin{equation}
  \Delta u = 0 \mbox{ \,\,in } \Omega \,\,\, {\rm  and} \,\,\, 
  \frac{\partial u}{\partial n} =\sigma \rho u \mbox{ on } {\partial \Omega}. 
\end{equation}
It has discrete spectrum
$$0=\sigma_0<\sigma_1(\Omega,\rho)\leq \sigma_2(\Omega,\rho)\leq
\sigma_3(\Omega,\rho)\leq\cdots\nearrow\infty.$$
Estimates  \eqref{weinstineq} and \eqref{ourineqst} can be generalized
to this setting.
\begin{theorem}
\label{steklovdensity}
  Let $\Omega$ be a simply-connected bounded Lipschitz planar domain with
  the density $\rho$ on the boundary.
  %$\rho\in L^\infty(\partial\Omega)$. 
  Then
  \begin{equation}\label{weinstineqDensity}
    \sigma_1(\Omega,\rho)\, M(\Omega) \leq 2\pi.
  \end{equation}
  \begin{equation}\label{ourineqstDensity}
    \sigma_2(\Omega,\rho) \, M(\Omega) < 4 \pi.
  \end{equation}
\end{theorem}
Inequality~\eqref{weinstineqDensity} was proved in
\cite{Weinst} similarly to \eqref{weinstineq}.  Estimate  \eqref{ourineqstDensity} was proved in \cite{GP}, and we refer to this paper for the details of the proof. In fact, it is almost identical  to that of  
\eqref{ourineqst}.

\subsection{Structure of the paper}
In section~\ref{confmapDisk} the Riemann mapping theorem is used
to transplant the Neumann and Steklov problems to the disk.
In section~\ref{sectionHerschOne},  we describe Hersch's  renormalization (\cite{Hersch}, see also \cite[p. 144]{SY}), which  is applied in section~\ref{proofWeinstock}  to prove Theorem~\ref{TheoremWeinstock}.
Theorem~\ref{TheoremSzego} and the first part of
Theorem~\ref{ThmNeumannDensity} are proved in section~\ref{SectionProofSzego}
using some results on subharmonic functions presented in section~\ref{SectionGrowth}.

In the remaining part  of the paper we  prove  estimates on the
second non-zero eigenvalues $\mu_2$ and $\sigma_2$. In
section~\ref{sectionHerschTwo}  some additional results on  Hersch's  renormalization are 
presented. In sections~\ref{hcaps}--\ref{SectionFolding} we introduce
hyperbolic caps and folded measures  (this idea  goes back to \cite{Nad}), 
which are  used to define test functions in section~\ref{SectionTestFunctions}. %These sections are based on ideas that go back to \cite{Nad}.  
In section~\ref{SectionInertia} we apply  a topological argument to prove
the existence of a suitable two-dimensional space of test
functions.  We use this argument in section \ref{sigma2} to  prove the first part of  Theorem~\ref{ourThmSteklov}, and in section \ref{subsecmu2} to prove inequality \eqref{ourineqneum} and the second part of
Theorem~\ref{ThmNeumannDensity}. Finally, in section ~\ref{SectionSharp} we  show that
inequalities~\eqref{ourineqneum} and~\eqref{ourineqst} are sharp. 
%and complete the proofs of Theorems \ref{ourThmNeumann} and \ref{ourThmSteklov}.

\section{Shape optimization for $\mu_1$ and $\sigma_1$}
\label{SectionShapeOne}
\subsection{Application of the Riemann mapping theorem}
\label{confmapDisk}
The results presented in the introduction share a common feature: they are
proved using methods of complex analysis. The following application of
the Riemann mapping theorem plays a key role in the proofs. We formulate it in such a way that it works for both Neumann 
and Steklov eigenvalue problems simulataneously.

Let $\Omega$ be a simply-connected bounded planar domain with Lipschitz
boundary. Consider a conformal diffeomorphism
$\phi:\mathbf{D}\rightarrow\Omega.$
Here and further on we  denote a conformal map and its extension to
the boundary by the same symbol.
Let $dz$ be the Lebesgue measure on $\Omega$, 
$ds$ be the arc-length measure on $\partial\Omega$ and let 
$\nu=\phi^*(dz)$ or $\nu=\phi^*(ds)$ be the pullback of either of these measures to $\overline{\mathbf{D}}$.
\begin{remark}
To simplify notation, we write  $\nu$ instead of $d\nu$ unless we integrate over this measure. The same convention applies to other 
measures defined later on.
\end{remark}
Set
\begin{gather}\label{rayleighAbstractDisk}
  \lambda_k(\nu)=\inf_{E} \sup_{0\neq f\in E}
  \frac{\int_{\overline{\mathbf{D}}}|\nabla f|^2\,dz}
  {\int_{\overline{\mathbf{D}}} f^2\,d\nu}, \,\,\,\,\, k=1,2,\dots
\end{gather}
The infimum is taken over all $k$-dimensional subspaces $E$ of the
Sobolev space $H^1(\mathbf{D})$ such that
\begin{gather}\label{admissibility}
  \int_{\overline{\mathbf{D}}}f\,d\nu=0\quad\mbox{ for all }f\in E.
\end{gather}

\bigskip

\bigskip

\begin{proposition}
  \begin{itemize}
  \item[]
  \item[--] For $\nu=\phi^*(dz)$,
    $\lambda_k(\nu)=\mu_k(\Omega)$ is the k-th eigenvalue of the
    Neumann problem.
  \item[--] For $\nu=\phi^*(ds)$,
    $\lambda_k(\nu)=\sigma_k(\Omega)$ is the k-th eigenvalue of the
    Steklov problem.
  \end{itemize}
\end{proposition}
\begin{proof}
  It is well known that the Dirichlet energy of a
  function $f$ is a conformal invariant in dimension two. The result then follows from
  the variational characterizations  of $\mu_k$ and $\sigma_k$ given by
  \eqref{RayleighNeumannOriginal} and  \eqref{RayleighSteklovOriginal}.
\end{proof}

The eigenvalue problems themselves can  be  pulled back to the
disk. The Neumann and Steklov problems on $\Omega$ are,
respectively,  equivalent to the following ones:
\begin{equation}\label{neumannDisk}
  -\Delta u =\mu \,|\phi'(z)|^2 \,u \mbox{\rm  \,\,in } \mathbf{D} \,\,\, {\rm and} \,\,\,
  \left.\frac{\partial u}{\partial r}\right |_{\mathbf{S}^1}=0, 
\end{equation}
\begin{equation}\label{steklovDisk}
  \Delta u = 0 \mbox{\rm  \,\, in } \mathbf{D} \,\,\, {\rm  and} \,\,\, \frac{\partial
    u}{\partial r} =\sigma |\phi'(z)| u \mbox{ on } {\mathbf{S}^1}. 
\end{equation}
This will be useful when treating the case of equality in
Szego's and Weinstock's inequalities.

\subsubsection*{Subharmonic functions}
Recall that a function $\delta:\Omega\rightarrow\mathbf{R}$ is
called \emph{subharmonic} if $\Delta\delta\geq 0$, and $\emph{log-subharmonic}$ if
$\Delta \log \delta \ge 0$.
We state here some simple facts about subharmonic functions that
will be used repeatedly.

\begin{lemma}\label{subharmSimple} Let $\delta \in C^2(\Omega)$ be a positive function.

\smallskip
  
(i) If $\delta$ is log-subharmonic, then it is subharmonic.
  
  \smallskip
  
 % \hskip2.3truecm
(ii) If $\delta$ is log--subharmonic and $\Delta\delta\leq 0$,  then 
  $\delta$ is constant.
\end{lemma}
\begin{proof}
  (i)
  It follows from
  \begin{gather}\label{laplogdelta}
    \Delta\log\delta=\frac{\delta\Delta\delta-|\nabla\delta|^2}{\delta^2}
  \end{gather}
  that
  $\Delta\delta=\delta\, \Delta\log\delta+
  \frac{|\nabla\delta|^2}{\delta}\geq 0.$\\
  
  \noindent
  (ii)
  From
  $|\nabla\delta|^2=\delta\Delta\delta-\delta^2\,\Delta\log\delta\leq
  0$ it follows that $|\nabla\delta|=0$, and hence $\delta$ is constant.
\end{proof}
\begin{lemma}\label{logharmonic}
  Let $\rho\in C^2(\Omega)$ be a positive function and  $\phi: \mathbf{D} \to \Omega$ be a conformal map.
  Consider the density $\rho(z)$ on $\Omega$,  and let 
  $$\delta(z)=\rho(\phi(z))|\phi'(z)|^2$$ be its pullback to the unit disk.
  Then the function $\log\delta$ is (sub)harmonic iff the function $\log\rho$ is
  (sub)harmonic.\end{lemma}
   In particular, if $\rho$ is constant  then $\log\delta$ is
  harmonic.
\begin{proof}
  The Gaussian curvature of the Riemannian metric $ds^2=\rho(dx^2+dy^2)$ is
  given by \eqref{isothermalGauss}.
  The pullback of $g$ by $\phi$ is
  $\delta(dx^2+dy^2)$
  where $\delta(z)=\rho(\phi(z))|\phi'(z)|^2.$
  Therefore,
  $K_{ \phi^*g}(z)=-\frac{1}{2\delta}\Delta\log\delta.$
  The result now follows from the formula
  $K_{ \phi^*g}(z)=K_g(\phi(z)).$
\end{proof}

\subsection{Hersch's  renormalization}\label{sectionHerschOne}
Let $\Psi:\overline{\mathbf{D}}\rightarrow\overline{\mathbf{D}}$ be a
diffeomorphism such that $\Psi(z)=z$ for each
$z\in\partial\mathbf{D}= \mathbf{S}^1.$
The center of mass relative to $\Psi$ of a finite measure
$\nu$ on the closed disk $\overline{\mathbf{D}}$ is defined by
$$\overline{\mathbf{D}}\ni \mbox{C}(\nu)=\frac{1}{M(\nu)}\int_{\overline{\mathbf{D}}}z\,d\,\Psi_*\nu$$
where $M(\nu)=\int_{\overline{\mathbf{D}}}\,d\nu$ is the mass of the
measure $\nu.$ Note that for $\nu~=~\phi^*(dz)$ and $\nu=\phi^*(ds)$
the mass of $\nu$ coincides with the mass of $\Omega$ defined by \eqref{masss}.
For example, the center of mass of the Dirac mass $\delta_p$
(where $p\in\overline{\mathbf{D}}$) is $\mbox{C}(\delta_p)=\Psi(p).$
Given $t\in\mathbf{R}^2$, define
$X_t:\overline{\mathbf{D}}\rightarrow\mathbf{R}$ by
\begin{gather}\label{Xs}
  X_t(z)=\langle\Psi(z),t\rangle.
\end{gather}

Note that:
\begin{itemize}
\item[--] For $\Psi=\id$, the functions $X_t$ are the
  eigenfunctions corresponding to the double eigenvalue
  $\sigma_1(\mathbf{D})=\sigma_2(\mathbf{D})=1$ of the Steklov problem.
\item[--] Let $J_1$ be the first Bessel function of the first kind, and
  let $\zeta \approx 1.84$ be the smallest positive zero of the
  derivative $J_1'$. Set
  $$f(r)=J_1(\zeta r)/J_1(\zeta).$$
  For $\Psi(re^{i\theta})=f(r)e^{i\theta}$, the functions $X_t$ are the
  eigenfunctions corresponding to the double eigenvalue
  $\mu_1(\mathbf{D})=\mu_2(\mathbf{D})$ of the Neumann problem.
\end{itemize}

Given $\xi\in\mathbf{D}$,  define the automorphism $d_\xi$ of $\mathbf{D}$ by
$$d_\xi(z)=\frac{z+\xi}{\overline{\xi}z+1}.$$
Observe that for any $p \in \partial \mathbf{D}$ and $-p \neq z \in \overline{\mathbf{D}}$, we have 
$\lim_{\xi \to p} d_\xi(z)=p$.
Then for any point $p\in\partial\mathbf{D}$,
\begin{gather}\label{conditionNiceMeasure}
    \lim_{\xi\rightarrow p}(d_\xi)_*\nu=\delta_p,
\end{gather}
%where $\delta_p$ is the Dirac mass at $p$.
if  the  measure $\nu$ is absolutely continuous with respect to the Lebesgue measure on $\mathbf{D}$ and with respect to the arc--length measure on 
$\partial \mathbf{D}$. Note that both measures defined in the beginning of section \ref{confmapDisk}
satisfy these conditions.

\begin{remark}
  We use the weak topology on the space of measures: a sequence
  of measure $(\nu_k)$ converges to $\nu$ iff for each continuous function $f$
  $$\lim_{k\rightarrow\infty}\int_{\overline{\mathbf{D}}}f\,d\nu_k
  =\int_{\overline{\mathbf{D}}}f\,d\nu.$$
  In particular,~\eqref{conditionNiceMeasure} means that
  $$\lim_{\xi\rightarrow p}\int_{\overline{\mathbf{D}}}(d_\xi)_*\nu
  =f(p)\quad \mbox{ for each } f\in C^0(\overline{\mathbf{D}}).$$
\end{remark}

\smallskip

\begin{proposition}\label{alaHersch} {\rm (cf. \cite{Hersch, SY, Gro, GNP})}
    Let $\nu$ be a finite measure on the closed disk
    $\overline{\mathbf{D}}$ satisfying~(\ref{conditionNiceMeasure}).
    Then there exists a point $\xi\in\mathbf{D}$ such that
    $$\int_{\overline{\mathbf{D}}}X_t\circ d_\xi\,d\nu=0\quad
    \mbox{ for each } t\in\mathbf{R}^2.$$
\end{proposition}
\begin{proof}
  Define the map $\Gamma:\overline{\mathbf{D}}\rightarrow\overline{\mathbf{D}}$ by
  $$\Gamma(\xi)=
  \begin{cases}
    \mbox{C}\bigl((d_\xi)_*\nu\bigr)& \mbox{ for } \xi\in\mathbf{D},\\
    \xi& \mbox{ for }\xi\in \partial\mathbf{D}.
  \end{cases}$$
  It follows from~(\ref{conditionNiceMeasure}) that for $p\in\partial\mathbf{D}$,
  \begin{align*}
    \lim_{\xi\rightarrow p}\mbox{C}\bigl((d_\xi)_*\nu\bigr)
    &=\mbox{C}(\delta_p)=\Psi(p)=p,
  \end{align*}
  so that $\Gamma$ is continuous. Moreover,  its restriction to $\partial\mathbf{D}$
  is the identity map. It follows by the standard topological  argument that $\Gamma$ is onto: there exists
  $\xi\in\mathbf{D}$ such that
  $$\mbox{C}\bigl((d_\xi)_*\nu\bigr)=0.$$
  Therefore,  for any $t\in\mathbf{R}^2$, 
  \begin{align*}
    \int_{\overline{\mathbf{D}}}X_t\circ d_\xi\,d\nu
    =\int_{\overline{\mathbf{D}}}\langle \Psi\circ d_\xi(z),t\rangle\,d\nu
    =M(\nu)\left\langle\mbox{C}\bigl((d_\xi)_*\nu\bigr),t\right\rangle=0.
  \end{align*}
\end{proof}

\subsection{Proof of Weinstock's theorem}
\label{proofWeinstock}
The goal of this section is to prove
Theorem~\ref{TheoremWeinstock}. Let $\phi:\mathbf{D}\rightarrow\Omega$
be a conformal equivalence. We will use the variational
characterization~\eqref{rayleighAbstractDisk} with the measure
$\nu=\phi^*(ds)$. This measure is supported on $\mathbf{S}^1$. We
use the test functions $X_t$ introduced in~\eqref{Xs} with
$\Psi(z)=z$, that is $X_t(z)=\langle z,t\rangle.$ Applying
Proposition~\ref{alaHersch}, we may assume that
\begin{gather*}
  \int_{\mathbf{S}^1}X_t\,d\nu=0\quad\mbox{ for all } t\in\mathbf{R}^2.
\end{gather*}

Choose  $s,t\in\mathbf{S}^1$ such that $\langle s,t\rangle=0.$
Observe that for any $z\in\mathbf{S}^1$,
$$X_{s}^2(z)+X_{t}^2(z)=1.$$
Switching $s$ and $t$ if necessary, we may assume  that
\begin{gather}
  \label{lowerboundSteklovX_t}
  \int_{\mathbf{S}^1}X_t^2\,d\nu\geq \frac{1}{2}\int_{\mathbf{S}^1}d\nu
  =\frac{M(\Omega)}{2}.
\end{gather}

Recall that $X_t$ is a Steklov eigenfunction
corresponding to the double eigenvalue $\sigma_1(\mathbf{D})=~1$. Therefore,
\begin{gather*}
  \int_{\mathbf{D}}|\nabla X_t|^2\,dz
  =\int_{\mathbf{S}^1}X_t^2\,ds=\pi.
\end{gather*}
Inequality~\eqref{weinstineq} then follows from the variational
characterization~(\ref{rayleighAbstractDisk}).

\subsubsection*{Case of equality}
Let $\Omega$ be such that $\sigma_1(\Omega)\,M(\Omega)=2\pi$. We may assume wihout loss of generality that
$M(\Omega)=2\pi$ (this can always be achieved using a  dilation).
For $t$ satisfying~\eqref{lowerboundSteklovX_t} we have
\begin{gather*}
  1=\sigma_1(\Omega)\, \leq
   \frac{\int_{\mathbf{D}}|\nabla X_t|^2\,dz}
  {\int_{\mathbf{S}^1}X_t^2\,d\nu}
  \leq 1=\sigma_1(\D).
\end{gather*}
It follows that $X_t$ is an eigenfunction of
problem~\eqref{steklovDisk} with eigenvalue~$1$:
\begin{equation}
  \Delta X_t = 0 \mbox{ in } \mathbf{D} \,\,\, {\rm  and} \,\,\,
  \frac{\partial X_t}{\partial r} =|\phi'(z)| X_t \mbox{ on }
  {\mathbf{S}^1}.
\end{equation}
However, by definition, $X_t$ also satisfies $\partial_r X_t=X_t$ so that
$|\phi'(z)|=1$ for each $z\in\mathbf{S}^1.$ 
By Lemma~\ref{logharmonic} the function $\log|\phi'(z)|$ is
harmonic. Because $\log|\phi'(z)|=0$ on $\mathbf{S}^1$, it is also
identically 0 on $\mathbf{D}$. Therefore, $|\phi'(z)|=1$ for each
$z\in\mathbf{D}$. It follows that $\phi:\mathbf{D}\rightarrow\Omega$
is an isometry.

\subsection{Growth of subharmonic functions}
\label{SectionGrowth}
Given a measure
$\nu=\delta(z)dz$ on $\mathbf{D}$, define
\begin{equation}
\label{zhe}
G(r)=\int_{B(0,r)}d\nu.
\end{equation}

\begin{lemma}\label{subharmonicgrowth}
  i) Let $\delta$ be  a positive subharmonic function on
  $\mathbf{D}$ such that  $G(1)=\pi$. Then
  $$G(r)\leq \pi r^2$$ for each $r\in [0,1]$.

\smallskip

  \noindent ii) The function $\delta$ is harmonic iff $G(r)=\pi r^2$ for each
  $r\in [0,1]$.
\end{lemma}
\begin{remark}
 Let  $\phi: \mathbf{D} \to \Omega$ and  $\delta(z)=|\phi'(z)|^2$. Then  Lemma \ref{subharmonicgrowth}
 states  that
  $$\mbox{Area}(\phi(B_r(0)))\leq\mbox{Area}(B_r(0)).$$
\end{remark}
\begin{proof}
i) 
  Let us write
  $$G(r)=\int_{B(0,r)}\delta(z)\,dz=\int_0^r W(\rho) \rho\, d\rho,$$
  where
  $$W(\rho)=\int_{0}^{2\pi}\delta(\rho e^{i\theta})\,d\theta.$$
  The function $W$ is non-decreasing on $[0,1]$ (see \cite{Levin}).
  %(see~\cite{Levin} for example).
  Indeed, define
  $\widetilde{W}:~\mathbf{D}~\rightarrow~\mathbf{R}$ by
  $$\widetilde{W}(z)=\int_{0}^{2\pi}\delta(ze^{i\theta})\,d\theta.$$
  This function is subharmonic and satisfies $\widetilde{W}(z)=W(|z|)$.
  It follows from the maximum principle that for any $z$ with $|z|=\rho$,
  $$W(\rho)=\widetilde{W}(z)\geq\max_{|z|\leq\rho}\widetilde{W}(z)
  =\max_{s\leq\rho}W(s).$$

  Therefore,
  \begin{align*}
    G(r)=r^2 \int_0^1 W(r\, \rho) \rho \, d\rho\leq r^2\int_0^1 W(\rho)\,\rho\,d\rho
    = G(1)r^2=\pi r^2.
  \end{align*}

  \noindent
  ii) If $\delta$ is harmonic, then the value of $\delta$ at
  zero is equal to the  average over any circle centered at zero:
  $$\delta(0)=\frac{1}{2\pi\rho}W(\rho).$$
  Hence,
  $G(r)=\int_{0}^1W(\rho)\rho\,d\rho=\pi\, \delta(0)\,r^2$, and 
 since  $G(1)=\pi$,  we get  $\delta(0)=1$ and $G(r)=\pi r^2$.
  
  In the opposite direction,  if $G(r)=\pi r^2$,  then $2\pi r=G'(r)=W(r)r$, so that
  the function $W$ is constant:
  $$2\pi=W(r)=\int_{0}^{2\pi}\delta(re^{i\theta})\,d\theta.$$
  Using a version of Jensen's formula
  (see~\cite[p.47]{Levin}) we get:
  \begin{gather*}
    \delta(0)+\int_{\mathbf{D}}\log\left(\frac{1}{|z|}\right)\Delta\delta(z)\,dz
    =\frac{W(r)}{2\pi}=1.
  \end{gather*}
  Since  the right--hand side is the average of $\delta$ over a circle of
  radius $r$, we get
  $$\delta(0)=\lim_{r\rightarrow 0}\frac{W(r)}{2\pi}=1.$$
  It follows that
  $$\int_{\mathbf{D}}\log\left(\frac{1}{|z|}\right)\Delta\delta(z)\,dz=0.$$
  Since  $\log\left(\frac{1}{|z|}\right)>0$ and $\Delta\,\delta\geq 0$, this  implies $\Delta\,\delta=0.$
\end{proof}
\begin{lemma}\label{comparisonLebesgue}
  Let $\delta(z)$ be a
  subharmonic function on $\mathbf{D}$, 
  $\nu=\delta(z)dz$ be the corresponding measure, and 
  $h:[0,1]\rightarrow\mathbf{R}$ be a smooth strictly increasing
  function with $h(0)=0$.  Suppose that $G(1)=\pi$, where $G(r)$ is
  given by \eqref{zhe}. Then 
  $$\int_{\mathbf{D}}h (|z|)\,d\nu\geq \int_{\mathbf{D}}h (|z|)\,dz.$$
  
\smallskip
\noindent Moreover, equality holds iff $\delta$ is harmonic.
\end{lemma}
\begin{proof}
  Using Lemma~\ref{subharmonicgrowth} and
  integration by parts we obtain:
  \begin{align*}
    \int_{\mathbf{D}}h(|z|)\,d\nu&=\int_{\mathbf{D}}h(|z|)\delta(z)\,dz
    =\int_{0}^1h (r)G'(r)\,dr\\
    &=h (1)G(1)-\int_{0}^1\frac{d}{dr}\bigl(h(r)\bigr)G(r)\,dr\\
    &\geq h(1)G(1)-\pi\int_{0}^1\frac{d}{dr}\bigl(h(r)\bigr)r^2\,dr\\
    &=2\pi\int_{0}^1h(r)r\,dr
    =\int_{\mathbf{D}}h(|z|)\,dz.
  \end{align*}
    
  If
  $\int_{\mathbf{D}}h(|z|)\,d\nu=\int_{\mathbf{D}}h(|z|)\,dz$,
  then from the computation above we deduce that
  $$\int_{0}^1\frac{d}{dr}\bigl(h(r)\bigr)\left(G(r)-\pi r^2\right)\,dr=0.$$
  By Lemma~\ref{subharmonicgrowth} (i), we have
  $G(r)\leq\pi r^2.$ Since $h$ is strictly increasing,
  % $\frac{d}{dr}\bigl(h(r)\bigr)>0$,  
  we get  $G(r)=\pi r^2$, which implies that $\delta$ is
  harmonic by Lemma~\ref{subharmonicgrowth}~(ii).
\end{proof}
\subsection{Proof of Szeg\"o's theorem}
\label{SectionProofSzego}
The goal of this section is to prove Theorem~\ref{TheoremSzego}.
Let $\phi:\mathbf{D}\rightarrow\Omega$
be a conformal equivalence. Let $\delta(z)=|\phi'(z)|^2$. It follows
from Lemma~\ref{logharmonic} that $\log\delta$ is harmonic, and hence
$\delta$ is subharmonic. Applying  a rescaling if necessary,
we may assume without loss of generality that
$M(\Omega)=\pi.$ We will use the variational
characterization~\eqref{rayleighAbstractDisk} with the measure 
$\nu=\phi^*(dz)=\delta dz$ and with the test functions
$$X_t(z)=\langle \Psi(z),t\rangle,$$
where $\Psi(re^{i\theta})=f(r)e^{i\theta}$, with
$f(r)=\frac{J_1(\zeta r)}{J_1(\zeta)}$.
By Proposition~\ref{alaHersch}, we  may assume
\begin{gather*}
  \int_{\mathbf{D}}X_t\,d\nu=0\quad\mbox{ for all } t\in\mathbf{R}^2.
\end{gather*}

Choose  $s,t\in\mathbf{S}^1$ such that $\langle s,t\rangle=0.$
Observe that for any $z\in\mathbf{D}$,
$$X_{s}^2(z)+X_{t}^2(z)=f^2(|z|).$$
Switching  $s$ and $t$ if necessary, we may assume that
\begin{align}\label{lowerboundNeumannX_s}
  \int_{\mathbf{D}}X_t^2\,d\nu&\geq\frac{1}{2}\int_{\mathbf{D}}f^2(|z|)\,d\nu(z)
  \geq \frac{1}{2}\int_{\mathbf{D}}f^2(|z|)\,dz,
\end{align}
where the last inequality follows from
Lemma~\ref{comparisonLebesgue}, because  $f(r)$ is strictly increasing on $[0,1]$.
Recall that the functions $X_t$ are
the Neumann eigenfunctions corresponding to the eigenvalue
$\mu_1(\mathbf{D})$.  Therefore, 
\begin{gather}\label{dirichletbound}
  \int_{\mathbf{D}}|\nabla X_t|^2\,dz
  =\mu_1(\mathbf{D})\int_{\mathbf{D}}X_t^2\,dz
  =\frac{\mu_1(\mathbf{D})}{2}\int_{\mathbf{D}}f^2(|z|)\,dz
\end{gather}
The proof of inequality~\eqref{ineqSzego} now follows from ~\eqref{lowerboundNeumannX_s},~\eqref{dirichletbound} and the
variational characterization~\eqref{rayleighAbstractDisk}.
\begin{remark}  This argument is motivated by  \cite[section 2.7]{GNP} and is a  modification of the proof  given in  \cite[p. 138]{SY}.  As indicated in  \cite[Remark 2.7.12]{GNP}, the novelty of our approach is  that it uses the properties of subharmonic functions.
 \end{remark} 

\begin{remark}\label{remarkBandle} The first part of Theorem \ref{ThmNeumannDensity} is proved in a similar way. To obtain inequality \eqref{ineqBandle}, we take 
  $\nu=\phi^*(\rho dz)$. By Lemma~\ref{logharmonic} this measure is
  also of the form $\delta\, dz$, where $\Delta\log\delta\geq 0$.
  It follows from Lemma~\ref{subharmSimple} that $\delta$ is subharmonic.
  The rest of the proof is unchanged.
\end{remark}

\subsubsection*{Case of equality}
Let us show  that  the  equality in~\eqref{ineqSzego} implies that
$\Omega$ is a disk. We will give two proofs of this fact.

\begin{proof}[First proof]
  Suppose that $\mu_1(\Omega)=\mu_1(\mathbf{D})$ and $M(\Omega)=\pi.$
  For the specific choice of $t$ made in~\eqref{lowerboundNeumannX_s}
  we have
  $$\mu_1(\D)=\mu_1(\Omega)\leq\frac{\int_{\mathbf{D}}|\nabla X_t|^2\,dz}
  {\int_{\mathbf{D}}X^2_t\,d\nu}\leq \mu_1(\mathbf{D})$$
  It follows that the function $X_t$ is a first  eigenfunction  of
  problem~\eqref{neumannDisk}:
  \begin{equation}
    -\Delta X_t =\delta\mu_1(\mathbf{D}) X_t \mbox{ in }\mathbf{D}.
  \end{equation}
  Because $-\Delta X_t=\mu_1(\mathbf{D}) X_t$, we deduce that
  $1=\delta=|\phi'(z)|^2$,  so that the conformal
  equivalence $\phi:\mathbf{D}\rightarrow\Omega$ is an isometry.
\end{proof}

Our  second proof is  a bit more involved, but it can be adapted to  the case of 
$\mu_2$: in section \ref{subsecmu2} we use a similar idea to prove that 
inequality \eqref{ourineqneum} is strict.
\begin{proof}[Second proof]
  Suppose that $\mu_1(\Omega)=\mu_1(\mathbf{D})$ and  $M(\Omega)=\pi.$
  For each $t \in \mathbf{S}^1$ we have
  \begin{align*}
    \mu_1(\mathbf{D})
    &\leq\frac{\int_{\mathbf{D}}|\nabla X_t|^2\,dz}
    {\int_{\mathbf{D}}X^2_t\,d\nu}
    =\mu_1(\mathbf{D})\frac{\int_{\mathbf{D}}X_t^2\,dz}
    {\int_{\mathbf{D}}X^2_t\,d\nu}.
  \end{align*}
  It follows that for each $t$
  \begin{gather*}
    \int_{\mathbf{D}}X^2_t\,d\nu\leq\int_{\mathbf{D}}X^2_t\,dz
    %=\frac{1}{2}\int_{\mathbf{D}}f^2(|z|)\,dz
  \end{gather*}
  Let $s,t\in\mathbf{S}^1$ be such that $\langle s,t\rangle=0$. It
  follows from $X_t^2(z)+X_s^2(z)=f^2(|z|)$ and the above inequality
  that
  \begin{gather*}
    \int_{\mathbf{D}}f^2(|z|)\,d\nu=\int_{\D}(X_t^2+X_s^2)\,d\nu
    \leq\int_{\D}(X_t^2+X_s^2)\,dz
    = \int_{\mathbf{D}}f^2(|z|)\,dz.
  \end{gather*}
  From Lemma~\ref{comparisonLebesgue} we get 
  $\int_{\mathbf{D}}f^2(|z|)\,d\nu=\int_{\mathbf{D}}f^2(|z|)\,dz$ and
  $\Delta\delta=0.$
  By construction of $\delta$ (see Lemma~\ref{logharmonic}) we have
  $\Delta\log\delta=0$,  so that by Lemma~\ref{laplogdelta} $\delta$
  is a constant. Therefore, since $M(\Omega)=\pi=M(\mathbf{D})$,  we have 
  $\delta(z)=|\phi'(z)|^2=1$, and hence
  $\phi:\mathbf{D}\rightarrow\Omega$ is an isometry.
\end{proof}

\section{Shape optimization for $\mu_2$ and $\sigma_2$}
\label{SectionShapeTwo}
\subsection{Hersch's method revisited}
\label{sectionHerschTwo}
In order to apply the Hersch method
to the second nonzero Neumann and Steklov  eigenvalues, more control is needed on the point $\xi$
obtained in  Proposition~\ref{alaHersch}.

\begin{proposition}
  Let $\nu$ be a finite measure on the closed disk
  $\overline{\mathbf{D}}$ satisfying~(\ref{conditionNiceMeasure}).
  The renormalizing point $\xi$ is unique and depends continuously
  on $\nu$.
\end{proposition}
\begin{proof}
  We give the proof for $\Psi(z)=z$ only. For more details on the
  general case, see~\cite{GNP}.
  First, suppose that $\mbox{C}(\nu)=0$ and let $\xi\neq 0$.
  Let $s=\frac{\xi}{|\xi|}.$ An easy computation shows that
  $\langle d_\xi(z), s\rangle > \langle z,s\rangle.$
  It follows that
  \begin{align*}
    \left\langle\mbox{C}\bigl((d_\xi)_*\nu\bigr),s\right\rangle
    =\frac{1}{M(\nu)}\int_{\overline{\mathbf{D}}}\langle d_\xi(z), s\rangle\,d\nu
    >\frac{1}{M(\nu)}\int_{\overline{\mathbf{D}}}\langle z, s\rangle\,d\nu
    =\left\langle\mbox{C}\bigl(\nu\bigr),s\right\rangle.
  \end{align*}
  In other words, if the center of mass of the measure $\nu$ is the origin,
  then $\xi=0$.

  Now, let $\nu$ be an arbitrary finite measure and suppose that it is
  renormalized by $d_\xi$ and $d_\eta$.
  By explicit computation  one gets
  $$d_{\eta}\circ d_{-\xi}=\frac{1-\eta \bar \xi}{1-\bar \eta \xi} d_\alpha,$$ where
  $\alpha=d_{-\xi}(\eta)$ and
  $\left|\frac{1-\eta \bar \xi}{1-\bar \eta\xi}\right|=1$.
  Moreover,
  \begin{align*}
    (d_\eta)_*\nu&=(d_\eta\circ d_{-\xi})_*\left(d_\xi\right)_*\nu
    =\frac{1-\eta \bar \xi}{1-\bar \eta \xi} (d_\alpha)_*\left(d_\xi\right)_*\nu.
  \end{align*}
  This implies that $d_\alpha$ renormalizes the measure
  $\left(d_\xi\right)_*\nu$ whose center of mass is already at
  the origin. It follows from the previous case that $\alpha=0$,
  which in turn implies $\eta=\xi.$

  Let us prove continuity.
  Let $(\nu_k)$ be a sequence of measures converging to the measure
  $\nu$. Without loss of generality suppose that $\nu$ is
  renormalized. Let $\xi_k\in\mathbf{D}\subset\overline{\mathbf{D}}$
  be the unique element such that $d_{\xi_k}$ renormalizes $\nu_k$.
  Let $(\xi_{k_j})$ be a convergent subsequence, say to
  $\xi\in\overline{\mathbf{D}}$.
  Now, by definition of $\xi_k$ there holds
  \begin{align*}
    0=\lim_{j\rightarrow\infty}\left|\int_{\overline{\mathbf{D}}}z\,\,(d_{\xi_{k_j}})_*d\nu_{k_j}\right|=
    \left|\int_{\overline{\mathbf{D}}}z\,\,(d_{\xi})_*d\nu\right|,
  \end{align*}
  and hence $d_\xi$ renormalizes $\nu$. Since we assumed that $\nu$
  is normalized, by uniqueness we get $\xi=0$. Therefore, $0$ is the
  unique accumulation point of the sequence $(\xi_k)$ in $\overline{\mathbf{D}}$
  and hence by compactness we get $\xi_k \to 0$.
\end{proof}
\subsection{Hyperbolic caps}
\label{hcaps}
Let $\gamma$ be a geodesic in the Poincar\'e disk
model, that is a diameter or the intersection of the disk with a
circle, which is orthogonal to $\mathbf{S}^1$.
\begin{figure}[h]
  \centering
  \psfrag{a}[][][1]{$a_{l,p}$}
  \psfrag{p}[][][1]{$p$}
  \psfrag{l}[][][1]{$l$}
  \includegraphics[width=7cm]{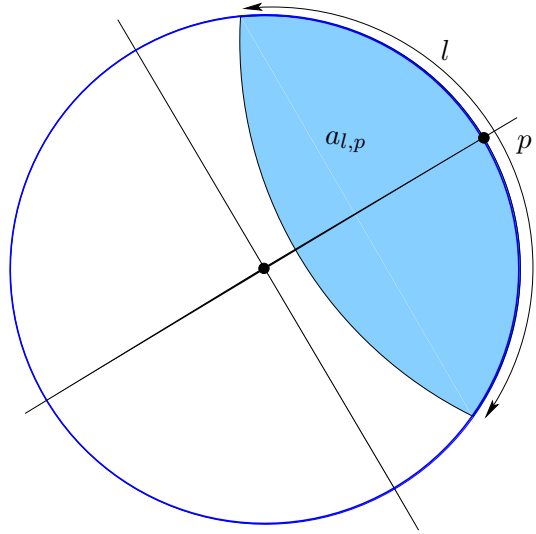}
  \caption{The hyperbolic cap $a_{l,p}$}
  \label{fighypocapo}
\end{figure}
Each connected component of $\mathbf{D}\setminus\gamma$
is called a \emph{hyperbolic cap}~\cite{GNP}.
Given $p\in\mathbf{S}^1$ and $l\in (0,2\pi)$, let $a_{l,p}$ be
the hyperbolic cap such that the circular segment $\partial
a_{l,p}\cap\mathbf{S}^1$ has length $l$ and is centered at $p$
(see Figure~\ref{fighypocapo}). This gives
an identification of the space  $\mathcal{HC}$ of all hyperbolic
caps with the cylinder $(0,2\pi)\times\mathbf{S}^1$.
Given a cap $a \in {\mathcal HC}$, let
$\tau_a:\overline{\mathbf{D}}\rightarrow\overline{\mathbf{D}}$ be
the reflection across the hyperbolic geodesic bounding $a$. That is,
$\tau_a$ is the unique non-trivial conformal involution of
$\mathbf{D}$ leaving every point of the geodesic
$\partial a\cap\mathbf{D}$ fixed. In particular the cap adjacent to $a$
is $a^*=\tau_a(a).$

\subsection{Folded  measure}
\label{SectionFolding}
The {\it lift} of a function
$u:\overline{a}\rightarrow\mathbf{R}$ is the function
$\tilde{u}:\overline{\mathbf{D}}\rightarrow\mathbf{R}$ defined by
\begin{equation}\label{lift}
  \tilde{u}(z)=
  \begin{cases}
    u(z) & \mbox{if } z\in\overline{a},\\
    u(\tau_az) & \mbox{if } z\in \overline{a^*}.
  \end{cases}
\end{equation}

%Here we should probably add the condition
% 1. nu is absolutely continuous w.r. to Lebesgue measure on D,
% 2. nu is absolutely continuous w.r. to arc-length  on S^1

As before, let $\nu$ be a finite measure on the closed disk $\overline{\mathbf{D}}$ which 
is absolutely continuous with respect to the Lebesgue measure on $\mathbf{D}$ and with respect to the arc--length measure on 
$\partial \mathbf{D}$. 
Observe that
\begin{align}\label{folding}
  \int_{\overline{\mathbf{D}}}\tilde{u}\,d\nu
  &=\int_{\overline{a}}u\,d\nu+\int_{\overline{a^*}}u\circ\tau_a\,d\nu\nonumber\\
  &=\int_{\overline{a}}u\,(d\nu+\tau_a^*d\nu).
\end{align}
The measure
\begin{equation}
\label{folded}
 d\nu_a =
\begin{cases}
  d\nu+\tau_a^*d\nu& \mbox{on } \overline{a},\\
  0 & \mbox{on } \overline{a^*}
\end{cases}
\end{equation}
is called the \emph{folded measure}. Equation~(\ref{folding}) can be rewritten
as
$$\int_{\overline{\mathbf{D}}}\tilde{u}\,d\nu=\int_{\overline{\mathbf{D}}}u\,d\nu_a.$$

\subsection{Test functions}
\label{SectionTestFunctions}
Let $a\in\mathcal{HC}$ be a hyperbolic cap and let
$\phi_a:\mathbf{D}\rightarrow a$ be a conformal equivalence.
For each $t\in\mathbf{R}^2$, define
$u_a^t:\overline{a}\rightarrow\mathbf{R}$ by
\begin{equation}\label{uat}
  u_a^t(z)=X_t\circ\phi_a^{-1}(z).
\end{equation}
For each cap $a\in\mathcal{HC}$ we will use the two-dimensional space of test functions
$$E_a=\left\{\tilde{u}_a^t\ :\ t\in\mathbf{R}^2\right\}$$
in the variational characterization~\eqref{rayleighAbstractDisk}.
It follows from the conformal invariance of
the Dirichlet energy that 
\begin{align}\label{doubleDirichlet}
  \int_{\mathbf{D}}|\nabla \tilde{u}_a^t|^2\,dz
  &=\int_{a}|\nabla u_a^t|^2\,dz
  +\int_{a^*}|\nabla(u_a^t\circ\tau_a)|^2\,dz\nonumber\\
  &=2\int_{a}|\nabla u_a^t|^2 \,dz
  =2\int_{\mathbf{D}}|\nabla X_t|^2 \,dz.
\end{align}
Observe that the denominator in~\eqref{rayleighAbstractDisk} can be
rewritten as 
\begin{gather}\label{intertia}
  \int_{\overline{\mathbf{D}}}(\tilde{u}_a^t)^2\,d\nu
  =\int_{\overline{\mathbf{D}}}X_t^2\,d\,\phi_a^*\,\nu_a.
\end{gather}
We call $\zeta_a=\phi_a^*\nu_a$ the \emph{rearranged measure}.
Taking~\eqref{doubleDirichlet} and~\eqref{intertia} into
account, we obtain from the variational characterization~\eqref{rayleighAbstractDisk} that
\begin{gather}\label{FoldedRayleigh}
\lambda_2(\nu)\leq 2\sup_{t\in\mathbf{S}^1}
  \frac{\int_{\overline{\mathbf{D}}}|\nabla X_t|^2\,dz}
  {\int_{\overline{\mathbf{D}}} X_t^2\,d\zeta_a}
\end{gather}
provided that the functions $\tilde{u}_a^t$ satisfy the admissibility
condition~\eqref{admissibility}.  This condition  can be rewritten in terms
of the rearranged measure $\zeta_a$:
\begin{gather*}
  \int_{\overline{\mathbf{D}}}\tilde{u}_a^t\,d\nu=\int_{\overline{\mathbf{D}}}X_t\,d\zeta_a = \,0.
\end{gather*}
In other words, ~\eqref{admissibility} is
satisfied by the function $\tilde{u}_a^t$ iff the rearranged measure
$\zeta_a$ is renormalized.

Note  that we are free to choose  the  conformal equivalences
$\phi_a:\mathbf{D}\rightarrow a$ in our construction of test functions.
\begin{lemma}\label{flipfloplemma}
  There exists a family of conformal equivalences
  $\left\{\phi_a:\mathbf{D}\rightarrow a\right\}_{a\in\mathcal{HC}}$
  such that the rearranged measure $\zeta_a$ depends continuously
  on the cap $a\in\mathcal{HC}$ and satisfies
  \begin{gather}
    \int_{\overline{\mathbf{D}}}X_t\,d\zeta_a=0,\label{normalization}\\
    \lim_{a\rightarrow \mathbf{D}}\zeta_a=\nu,\label{fulldisk}\\
    \lim_{a\rightarrow p}\zeta_a=R_p^*d\nu\label{flipflopProperty},
  \end{gather}
  where $p \in\mathbf{S}^1$ and
  $R_p(x)=x-2\langle x, p\rangle$ is the reflection with
  respect to the diameter orthogonal to the vector $p$.
\end{lemma}
From now on, we fix the family of conformal maps $\phi_a$ defined in
Lemma~\ref{flipfloplemma}.
\begin{proof}[Proof of Lemma~\ref{flipfloplemma}]
    Let us give an outline of the proof, for more details, see~\cite[Section 2.5]{GNP}.
    Start with any continuous family of conformal maps
    $\left\{\psi_a:\mathbf{D}\rightarrow a\right\}_{a\in\mathcal{HC}}$,
    such that $\lim_{a\rightarrow\mathbf{D}}\psi_a=\id$.
    The maps $\phi_a$ are defined by composing the $\psi_a$'s on both
    sides with automorphisms of the disk
    appearing in the Hersch renormalization procedure.
    In particular,~\eqref{normalization} is automatically satisfied.
    As the cap $a$ converges to the full disk $\mathbf{D}$, the
    conformal equivalences $\phi_a$ converge
    to the identity map on $\mathbf{D}$, which implies~\eqref{fulldisk}.
    Finally, setting $n=1$ in~\cite[Lemma~4.3.2]{GNP} one gets \eqref{flipflopProperty}.
\end{proof}

\subsection{Maximization of the moment of inertia}
\label{SectionInertia}
The \emph{moment of inertia} of a finite measure $\nu$ on the closed
disk $\overline{\D}$ is
the quadratic form $V_\nu:\mathbf{R}^2\rightarrow\mathbf{R}$ defined
by
$$V_\nu(t)=\int_{\overline{\mathbf{D}}}X_t^2\,d\nu,$$
where $X_t$ is defined by \eqref{Xs}.
When $\Psi=\id$ and $t\in\mathbf{S}^1$ this corresponds to the usual
definition given in mechanics for the moment of inertia of $\nu$ with
respect to the axis orthogonal to $t$.

Let $\mathbf{R}P^1=\mathbf{S}^1/\mathbf{Z}_2$ be the projective line. We
denote by $[t]\in \mathbf{R}P^1$ the element of the projective line
corresponding to the pair of points $\pm t \in\mathbf{S}^1$. We say that
$[t] \in \mathbf{R}P^1$ is a {\it maximizing direction} for the measure $\nu$
if $V_\nu(t)\geq V_\nu(s)$ for any $s \in \mathbf{S}^1$. The measure $\nu$
is called {\it simple} if there is a unique maximizing
direction. Otherwise, it is said to be {\it multiple}.

\begin{lemma}
  A measure $\nu$ is
  multiple if and only if $V_\nu(t)$ does not depend on
  $t \in\mathbf{S}^1$.
\end{lemma}
\begin{proof}
  This follows from the fact that $V_\nu$ is quadratic, see \cite[Lemma 2.6.1]{GNP}.
\end{proof}

\begin{proposition}\label{multexist}
  If the measure $\nu$ is simple, then there exists
  a cap $a\in \mathcal{HC}$ such that the rearranged measure $\zeta_a$
  is multiple.
\end{proposition}
Proposition \ref{multexist} is proved  by contradiction. Assume that the measure $\nu$, as
well as the rearranged measures $\zeta_a$ for all $a\in {\mathcal HC}$, are simple. Given a hyperbolic cap
$a$,  let $m(a)\in\mathbf{R}P^1$ be the unique
maximizing direction for $\zeta_a$.
Let us understand the behavior of the maximizing direction as the cap
$a$ degenerates either to the full disk or to a point.
\begin{lemma}\label{degeneratecaps}
  Let  the measure $\nu$ as well as the rearranged measures
  $\zeta_a$ for all $a\in {\mathcal HC}$  be simple.
  Then
  \begin{gather}
    \lim_{a\rightarrow\mathbf{D}} m(\zeta_a)=m(\nu)\label{eqndegenerate1}\\
    \lim_{a\rightarrow e^{i\theta}} m(\zeta_a)=[e^{2i\theta}]\label{eqndegenerate2}.
  \end{gather}
\end{lemma}
\begin{proof}
  Without loss of generality, assume $m[\nu]=[e_1]$.
  First, note that formula~\eqref{eqndegenerate1} immediately follows from
  \eqref{fulldisk}. Let us prove \eqref{eqndegenerate2}. Set $p=e^{i\theta}$.
  Formula~\eqref{flipflopProperty} implies
  \begin{equation}\label{Xps}
    \lim_{a\rightarrow p}\int_{\overline{\mathbf{D}}}X_t^2\,d\zeta_a=
    \int_{\overline{\mathbf{D}}}X_t^2\,R_p^*d\nu
    =\int_{\overline{\mathbf{D}}}X_t^2\circ R_p\,d\nu
    =\int_{\overline{\mathbf{D}}}X_{R_pt}^2\,d\nu.
  \end{equation}
  Since $\nu$ is simple, $m(\nu)=[e_1]$ is the unique maximizing direction for $\nu$
  and the right hand side of \eqref{Xps} is maximal for
  $R_pt=\pm e_1$. Applying $R_p$ on both sides we get
  $t=\pm e^{2i\theta}$.
\end{proof}
\begin{proof}[Proof of Proposition~\ref{multexist}]
  Suppose that for each hyperbolic cap $a\in\mathcal{HC}$, the rearranged
  measure $\zeta_a$ is simple. Recall that the space $\mathcal{HC}$ is
  identified with the open cylinder $(0,2\pi)\times\mathbf{S}^1$.
  Define
  $h:(0,2\pi)\times S^1\rightarrow \mathbf{R}P^1$
  by $h(l,p)=m(a_{l,p}).$ The maximizing
  direction depends continuously on the cap $a$.
  Therefore, it follows from Lemma~\ref{degeneratecaps} that $h$ extends to a continuous
  map on the closed cylinder $[0,2\pi]\times S^1$ such that
  $$h(0,e^{i\theta})=[e_1],\ h(2\pi,e^{i\theta})=[e^{2i\theta}].$$
  This means that $h$ is a homotopy between a trivial loop and a
  non-contractible loop in $\mathbf{R}P^1$. This is a contradiction.
\end{proof}

\subsection{Estimate on $\sigma_2$}
\label{sigma2}
In this section we prove Theorem~\ref{ourThmSteklov}. Consider  the functions $X_t$ introduced in~\eqref{Xs} with
$\Psi(z)=z$, that is $X_t(z)=\langle z,t\rangle.$ 
The measure $\nu=\phi^*(ds)$ is supported on $\mathbf{S}^1$.
We provide details only in the case when the measure $\nu$ is
simple. If the measure $\nu$ is multiple the proof is easier,
see~\cite{GP}.

Let $a\in\mathcal{HC}$ be a cap
such that the rearranged measure $\zeta_a$ is multiple.
Using~\eqref{FoldedRayleigh} and taking into account that the
functions $X_t$ are eigenfunctions corresponding to
$\sigma_1(\mathbf{D})=1$, we get
\begin{align}
\label{al2}
  \sigma_2(\Omega)
  \leq
  2\frac{\int_{\overline{\mathbf{D}}}|\nabla X_t|^2\,dz}
  {\int_{\mathbf{S}^1} X_t^2\,d\zeta_a}
  =
  2\frac{\int_{\mathbf{S}^1}X_t^2\,ds}
  {\int_{\mathbf{S}^1} X_t^2\,d\zeta_a}
  =\frac{2\pi}
  {\int_{\mathbf{S}^1} X_t^2\,d\zeta_a}
  \end{align}
  Given $t\in\mathbf{S}^1$, choose $s\in\mathbf{S}^1$ such that
  $\langle t,s\rangle=0.$ Multiplicity of the rearranged measure $\zeta_a$
  and $X_t^2+X_s^2=1$ on $\mathbf{S}^1$ implies
  \begin{align}
  \label{al}
    \int_{\mathbf{S}^1}X_t^2\,d\zeta_a
    =\frac{1}{2}\int_{\overline{\mathbf{D}}}(X_t^2+X_s^2)\,d\zeta_a(z)
    =\frac{1}{2}M(\Omega)
  \end{align}
  This proves that
  $\sigma_2(\Omega)\leq \frac{4\pi}{M(\Omega)}.$
  
  \subsubsection*{The inequality is strict}
  Let $w_a^t\in C^{\infty}(\mathbf{D})$ be the unique
  harmonic extension of $\tilde{u}_a^t\bigl|\bigr._{\mathbf{S}^1}$, that is
  \begin{gather}
    \label{wat}
    \begin{cases}
      \Delta w_a^t=0& \mbox{ in } \mathbf{D},\\
      w_a^t=\tilde{u}_a^t& \mbox{ on }\mathbf{S}^1.
    \end{cases}
  \end{gather}
  These functions are smooth while  the original test functions
  $u_a^t$ are not smooth along the geodesic bounding the hyperbolic
  cap $a$ (see \cite[Lemma 3.4.1]{GP}). Therefore, $w_a^t\neq\tilde{u}_a^t$ in $H^1(\mathbf{D})$.
  It is well-known that a harmonic function, such as $w_a^t$, is the
  unique minimizer of the Dirichlet energy among all functions with
  the same boundary data~(see \cite[p. 157]{JST}).  Therefore,
  \begin{align}
 \label{al1}
    \int_{\mathbf{D}}|\nabla w_a^t|^2\,dz&<
    \int_{\mathbf{D}}|\nabla\tilde{u}_a^t|^2\,dz.
  \end{align}
Let us take  the functions $w_a^t$   as test functions instead of $\tilde{u}_a^t$
in section~\ref{SectionTestFunctions}. 
Their admissibility follows from \eqref{normalization}, because
  $w_a^t=\tilde{u}_a^t$  on $\mathbf{S}^1$. For the same reason,  the denominator in the Rayleigh   quotient calculated in \eqref{al} remains unchanged.  Together with \eqref{al1}  this implies  that
  inequality~\eqref{al2} is strict.
  %The rest of the proof is exactly the same.
  
\subsection{Estimate on $\mu_2$}
\label{subsecmu2}
We use the measure $\nu=\phi^*(dz)$ and the functions
$X_t(z)=\langle \Psi(z),t\rangle$,  where
$\Psi(re^{i\theta})=f(r)e^{i\theta}$, with
$f(r)=\frac{J_1(\zeta r)}{J_1(\zeta)}$.

\begin{lemma}\label{subharm}
  The rearranged measure $\zeta_a$ on $\mathbf{D}$ can
  be represented as $\zeta_a=\delta(z)dz$, where
  $\delta:\mathbf{D}\rightarrow\mathbf{R}$ is a subharmonic function.
\end{lemma}
\begin{proof}
  The rearranged measure $\zeta_a=\phi_a^*(\nu_a)$ can be rewritten as
  $$\zeta_a=(\phi\circ\phi_a)^*dz+(\phi\circ\tau_a\circ\phi_a)^*(dz)
  =\alpha(z)dz+\beta(z)dz$$
  where $\alpha(z)=|(\phi\circ\phi_a)'(z)|^2$ and 
  $\beta(z)=|(\phi\circ\tau_a\circ\phi_a)'(z)|^2.$
  It follows from Lemma~\ref{logharmonic} that
  $\log\alpha$ and $\log\beta$ are harmonic functions. 
  Therefore, $\alpha(z)$ and $\beta(z)$ are
  subharmonic by Lemma~\ref{subharmSimple}.
\end{proof}

\begin{proof}[Proof of inequality~\eqref{ourineqneum}]
  We provide details only in the case when the measure $\nu$ is
  simple. If the measure $\nu$ is multiple, then the proof is easier,
  see~\cite{GNP}.

  Without loss of generality, suppose that $M(\Omega)=\pi.$
  Let $a\in\mathcal{HC}$ be a cap
  such that the rearranged measure $\zeta_a$ is multiple.
  Using~\eqref{FoldedRayleigh} and taking into account that the
  functions $X_t$ are eigenfunctions corresponding to
  $\mu_1(\mathbf{D})$, we get
  \begin{align*}
    \mu_2(\Omega)
    \leq
    2\frac{\int_{\overline{\mathbf{D}}}|\nabla X_t|^2\,dz}
    {\int_{\overline{\mathbf{D}}} X_t^2\,d\zeta_a}
    =
    2\mu_1(\mathbf{D})\frac{\int_{\overline{\mathbf{D}}}X_t^2\,dz}
    {\int_{\overline{\mathbf{D}}} X_t^2\,d\zeta_a}
    =\mu_1(\mathbf{D})\frac{\int_{\overline{\mathbf{D}}}f^2(|z|)\,dz}
    {\int_{\overline{\mathbf{D}}} X_t^2\,d\zeta_a}
  \end{align*}
  Given $t\in\mathbf{S}^1$, choose $s\in\mathbf{S}^1$ such that
  $\langle t,s\rangle=0.$ Multiplicity of the rearranged measure
  $\zeta_a$  implies
  \begin{align*}
    \int_{\overline{\mathbf{D}}}X_t^2\,d\zeta_a
    =\frac{1}{2}\int_{\overline{\mathbf{D}}}(X_t^2+X_s^2)\,d\zeta_a
    =\frac{1}{2}\int_{\overline{\mathbf{D}}}f^2(|z|)\,d\zeta_a
  \end{align*}
  This leads to
  \begin{align}\label{robolomoulo}
    \mu_2(\Omega)
    \leq
    2\mu_1(\mathbf{D})
    \frac{\int_{\overline{\mathbf{D}}}f^2(|z|)\,dz}
    {\int_{\overline{\mathbf{D}}}f^2(|z|)\,d\zeta_a}
  \end{align}
  By Lemma \ref{subharm},  one can apply Lemma ~\ref{comparisonLebesgue} to the measure
  $\zeta_a$.  Hence, \eqref{robolomoulo}  implies 
  \begin{equation}
  \label{nonstr}
  \mu_2(\Omega)\leq 2\mu_1(\mathbf{D}).
  \end {equation}

  \subsubsection*{Proof of Theorem \ref{ThmNeumannDensity}} Inequality \eqref{ineqBandle} was already proved in Remark \ref{remarkBandle}.
    The proof of inequality~\ref{ourineqneumDensity} is almost identical to
    the one above. We use the measure $\nu=\phi^*(\rho dz)$. By Lemma~\ref{logharmonic} this measure is
    also of the form $\delta dz$ for some subharmonic function $\delta$.
    The rearranged measure $\zeta_a=\phi_a^*(\nu_a)$ can be rewritten
    $$\zeta_a=(\phi\circ\phi_a)^*(\rho \,dz)+
    (\phi\circ\tau_a\circ\phi_a)^*(\rho \,dz)=\alpha(z)dz+\beta(z)dz$$
    where (by Lemma~\ref{logharmonic}) $\alpha(z)$ and
    $\beta(z)$ are subharmonic. Hence, the statement of  Lemma~\ref{subharm} holds  in this case as well. 
    The rest of the proof is unchanged. \qed

  \subsubsection*{The inequality  \eqref{nonstr} is strict}
  %In order to prove this it is sufficient to show that \eqref{robolomoulo} is strict. 
  Suppose that $\mu_2(\Omega)=2\mu_1(\mathbf{D})$.
  Then, by ~\eqref{robolomoulo} we get 
  \begin{gather*}
    \int_{\mathbf{D}}f^2(|z|)\,d\zeta_a\leq\int_{\mathbf{D}}f^2(|z|)\,dz.
  \end{gather*}
  Recall that  according to Lemma~\ref{subharm},  $\zeta_a=\delta(z)dz$ for some
  subharmonic function $\delta$.
  It follows from Lemma~\ref{comparisonLebesgue} that
  $\int_{\mathbf{D}}f^2(|z|)\,d\zeta_a=\int_{\mathbf{D}}f^2(|z|)\,dz$
  and that $\Delta\delta=0$.
  Now, by construction of $\delta$ in the proof of Lemma~\ref{subharm}
  we have 
  $\delta=\alpha+\beta$ with
  $\Delta\log\alpha=0$ and $\Delta\log\beta=0$.
  It follows from $\Delta\alpha\geq 0$, $\Delta\beta\geq 0$ and from
  $0=\Delta\delta=\Delta\alpha+\Delta\beta$
  that $\Delta\alpha=0$ and $\Delta\beta=0.$
  Hence, by Lemma~\ref{subharmSimple} (ii) the functions $\alpha$ and
  $\beta$ are constant. Now, from the proof of Lemma~\ref{subharm} we
  see that 
  $\alpha(z)=|(\phi\circ\phi_a)'(z)|^2$ and 
  $\beta(z)=|(\phi\circ\tau_a\circ\phi_a)'(z)|^2.$
  This implies that $\phi\circ\phi_a$ and $\phi\circ\tau_a\circ\phi_a$
  are dilations.
Recall that $a=\phi_a(\mathbf{D})$ and $a^*=\tau_a\circ\phi_a(\mathbf{D})$, where $a^*$  is the cap adjacent to $a$.
Hence,  $\phi(a)$ and $\phi(a^*)$ are disjoint disks. 
We get a contradiction, because
  \begin{gather*}
    \overline{\Omega}=
    \overline{\phi(a)}
    \cup
    \overline{\phi(a^*)}
    %=
    %\overline{\phi\circ\phi_a(\mathbf{D})}
    %\cup
    %\overline{\phi\circ\tau_a\circ\phi_a(\mathbf{D})}
  \end{gather*}
  is a connected set. This completes the proof of the first
  part of  Theorem \ref{ourThmNeumann}.
\end{proof}

\subsection{The inequalities  for $\mu_2$ and $\sigma_2$ are sharp}
\label{SectionSharp}
The goal of this section is to prove the second parts of
Theorems~\ref{ourThmNeumann} and~\ref{ourThmSteklov}.

\subsubsection*{Neumann boundary conditions}
The family  $\Omega_\varepsilon$ is constructed by joining two disks using
a thin passage.
\begin{figure}[h]
  \centering
  \psfrag{e}[][][1]{L}
  \psfrag{f}[][][1]{$\varepsilon$}
  \includegraphics[width=11cm]{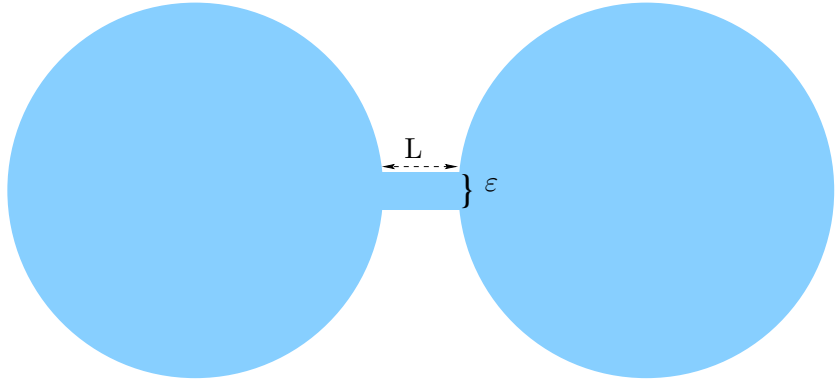}
  \caption{The domain $\Omega_\varepsilon$}
\end{figure}
More precisely,  let
$\Omega_\varepsilon=\D_1\cup P_\varepsilon \cup \D_2$,
where $\D_1$ and $\D_2$ are two
copies of the unit disk joined  by a rectangular passage
$P_\varepsilon$ of length $L$ and width $\varepsilon$.
It follows from~\cite{JM} (see also \cite{Anne, HSS}) that
the Neumann spectrum of $\Omega_\varepsilon$ converges to the disjoint
union of the Neumann spectra of $\D_1$ and  $\D_2$  and  the
Dirichlet spectrum of the operator $-\frac{d^2}{dx^2}$ acting on the
interval $[0,L]$. The first Dirichlet eigenvalue of $[0,L]$ is
$\frac{\pi^2}{L^2}$. It follows that for $L<1$ we have
\begin{gather*}
  \lim_{\varepsilon\rightarrow 0} \mu_0(\Omega_\varepsilon)=0,\
  \lim_{\varepsilon\rightarrow 0} \mu_1(\Omega_\varepsilon)=0,\\
  \lim_{\varepsilon\rightarrow 0} \mu_2(\Omega_\varepsilon)=\mu_1(\mathbf{D}).
\end{gather*}
Since $\lim_{\varepsilon\rightarrow 0}M(\Omega_\varepsilon)=2\pi$,
this completes  the proof.

\subsubsection*{Steklov boundary conditions}
The details of the proof can be found in~\cite{GP}. Let us
mention that simply joining two disks by a thin
passage does not work in the case of Steklov eigenvalues.  In fact, it was proved in~\cite{GP} that for the domains
$\Omega_\varepsilon$ defined above, the Steklov spectrum is
collapsing:
\begin{gather}
  \lim_{\varepsilon\rightarrow 0}\sigma_k(\Omega_\varepsilon)=0
  \quad \mbox{ for each } k=1,2,\dots.
\end{gather}
Instead, we use a family of domains  $\Sigma_{\varepsilon}$ , $\varepsilon \to 0+$, 
obtained by ``pulling two disks apart'' as shown  on Figure 3.  Similarly, taking  $k$ disks
pulled apart,  we show in \cite{GP}  that the Hersch--Payne--Schiffer inequality \eqref{hps} is sharp for all $k\ge 1$. 
\begin{figure}[h]
  \centering
  \psfrag{e}[][][1]{$\varepsilon$}
  \includegraphics[width=11cm]{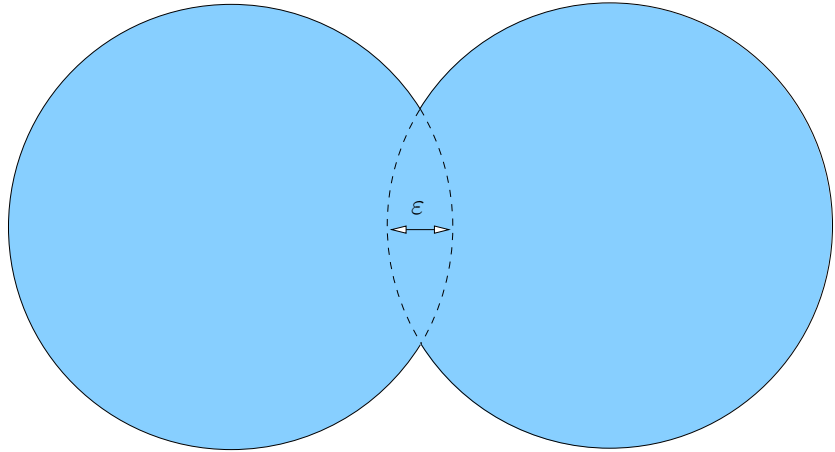}
  \caption{The domain $\Sigma_\varepsilon$}
  \label{figpullappart}
\end{figure}

\subsection*{Acknowledgments}  
Many ideas presented in this survey originated in 
\cite{GNP}, and we are grateful to Nikolai Nadirashvili for sharing with us his illuminating insights. 
%that he shared with us. %during the work on that paper. 
We are also thankful to Michael Levitin, Marco Marletta and 
Mikhail Sodin for useful discussions.


\begin{thebibliography}{DDDD}

\bibitem[Ann]{Anne} {\scshape Ann\'e, C.},
{\em A note on the generalized dumbbell problem},
Proc. Amer. Math. Soc. \textbf{123} (1995), no. 8, 2595--2599.

%bibitem[AB]{AB} {\scshape Ashbaugh, M. and Benguria, R.},
%{\em Isoperimetric inequalities for eigenvalues of the Laplacian},
%Spectral theory and mathematical physics: a Festschrift in honor of
%Barry Simon's 60th birthday, 105--139, Proc. Sympos. Pure Math.,
%\textbf{76}, Part 1, Amer. Math. Soc., Providence,  2007.



\bibitem[Ba1]{Bandle1} {\scshape Bandle, C.}, {\em Isoperimetric
inequalities and applications}, Pitman, Boston, 1980.

\bibitem[Ba2]{Bandle2} {\scshape Bandle, C.},
  {\em Isoperimetric inequality for some eigenvalues of an inhomogeneous, free membrane},
  SIAM J. Appl. Math. \textbf{22} (1972), 142--147.

\bibitem[Ba3]{Bandle3} {\scshape Bandle, C.}, {\em \"Uber des
Stekloffsche Eigenwertproblem: Isoperimetrische Ungleichungen f\"ur
symmetrische Gebiete.} Z. Angew. Math. Phys. \textbf{19} (1968),
627-237.



\bibitem[Bro]{Brock} {\scshape Brock, F.}, {\em An isoperimetric
inequality for eigenvalues of the Stekloff problem}, Z. Angew. Math.
Mech. \textbf{81} (2001), 69-71.

%\bibitem[BB]{BB} {\scshape Bucur D. and Buttazzo G.}, {\em Variational
%methods in shape optimization problems}, Birkh\"auser, Boston, 2005.

%\bibitem[BH]{BH} {\scshape Bucur D. and Henrot A.}, {\em Minimization
%of the third eigenvalue of the Dirichlet Laplacian}, Proc. Roy. Soc.
%London, Ser. A, \textbf{456} (2000), 985-996.

%\bibitem[BR]{BR} {\scshape Beckenbach, E. and Rad\'o, T.},
%\emph{Subharmonic functions and surfaces of negative curvature}, Trans.
%Amer. Math. Soc. \textbf{35}, no. 3 (1933), 662-674.

\bibitem[Ch]{Chavel} I. Chavel. {\em Eigenvalues in Riemannian geometry}, Academic Press, 1984.

\bibitem[CE]{ElCol}
  {\scshape Colbois, B. and El Soufi, A.},
  \emph{Extremal eigenvalues of the {L}aplacian in a conformal class
    of metrics: the `conformal spectrum'},
  Ann. Global Anal. Geom. (4) \textbf{24},
  (2003), 337--349.


%\bibitem[Co]{Co}{\scshape Courant, R.}, {\em Beweis des Satzes, da\ss von allen homogenen Membranen gegebenen
%Umfanges und gegebener Spannung die kreisf\"rmige den tiefsten
%Grundton besitzt}, Math. Zeit. \textbf{1}. No. 2-3 (1918), 321-328.

%\bibitem[DZ]{DZ} {\scshape Delfour, M. and Zolesio, J.-P.}, {\em Shapes and geometries.
%Analysis, differential calculus, and optimization},  Advances in
%Design and Control, 4. SIAM, Philadelphia, 2001.

%\bibitem[Din]{Di} {\scshape Ding, Z.}, {\em A Proof of the Trace Theorem of Sobolev Spaces on
%Lipschitz Domains}, Proc. Amer. Math. Soc. \textbf{124}, No. 2
%(1996), 591-600.
%
\bibitem[Dit]{Dittmar} {\scshape Dittmar, B.}, {\em Sums of
reciprocal Stekloff eigenvalues}, Math. Nachr. \textbf{268} (2004),
44-49.

\bibitem[DFN]{DFN} {\scshape Dubrovin, B., Fomenko A. and  Novikov S.}, {\em Modern geometry:
methods and applications,} vol. 1, Springer, 1992.

\bibitem[Ed]{Edward} {\scshape Edward, J.}, {\em An inequality for
Steklov eigenvalues for planar domains}, Z. Angew. Math. Phys.
\textbf{45} (1994), 493-496.

\bibitem[Fa]{Faber}
  {\scshape Faber, G.},
  \emph{Beweis, dass unter allen homogenen Membranen von gleicher
    Fl\"ache und gleicher Spannung die kreisf\"ormige den tiefsten
    Grundton gibt},
  Sitzungberichte der mathematisch-physikalischen Klasse der
  Bayerischen Akademie der Wissenschaften zu M\"unchen Jahrgang, (1923),
  169--172.

%\bibitem[FK]{FK} {\scshape Fox, D. and Kuttler, J.}, {\it Sloshing
%frequencies}, Z. Angew. Math. Phys. 34 (1983), no. 5, 668--696.

\bibitem[GNP]{GNP} {\scshape Girouard, A., Nadirashvili, N. and
Polterovich, I.}, {\em Maximization of the second positive Neumann
eigenvalue for planar domains}, arXiv:0803.4171, 1--24.


\bibitem[GP]{GP} {\scshape Girouard, A. and Polterovich, I.},
  {\em On the Hersch--Payne--Schiffer inequalities for Steklov eigenvalues}, arXiv:0808.2968, 1--15.% to appear in Funct. Anal. Appl.

%\bibitem[GP2]{GP1}  {\scshape Girouard, A. and Polterovich, I.},
%  {\em A note on the }, in preparation.

%\bibitem[GR]{GR} {\scshape Gradshteyn, I. and Ryzhik, I.},
%\emph{Table of integrals, series and products}, Sixth Edition,
%Academic Press, San Diego, 2000.

\bibitem [Gr]{Gro} {\scshape Gromov, M.}, {\em Metric invariants of Kähler manifolds},  
Differential geometry and topology (Alghero, 1992),  90--116, World Sci. Publ., River Edge, NJ, 1993. 

\bibitem[HSS]{HSS} {\scshape Hempel, R., Seco, L. and  Simon, B.},
{\em The essential spectrum of Neumann Laplacians on some bounded
singular domains}, J. Funct. Anal. \textbf{102} (1991), no. 2, 448--483.

\bibitem[Hen]{Henrot}
  {\scshape Henrot, A.},
  \emph{Extremum problems for eigenvalues of elliptic operators},
  Birkh\"auser Verlag, Basel, 2006.

%\bibitem[HP]{HP} {\scshape Henrot,  A. and Pierre, M.}, {\em
%Variation et optimisation de formes}, Springer, Berlin, 2005.

\bibitem[HePhSa]{HPS2} {\scshape Henrot, A., Philippin, G. and
Safoui, A.}, {\em Some isoperimetric inequalities with application
to the Stekloff problem},   J.  Convex  Anal. \textbf{15} (2008), no. 3,  581-592.

\bibitem[Her]{Hersch}
  {\scshape Hersch, J.}
  \emph{Quatre propri\'et\'es isop\'erim\'etriques de membranes
    sph\'eriques homog\`enes},
  C. R. Acad. Sci. Paris S\'er. A-B \textbf{270}, (1970),
  A1645--A1648.

\bibitem[HP]{HP}  {\scshape Hersch, J. and  Payne, L.}, {\em Extremal principles and isoperimetric inequalities for some mixed problems of Stekloff's type}, Z. Angew. Math. Phys. \textbf{19} (1968), 802-817.

\bibitem[HPS]{HPS1} {\scshape Hersch, J., Payne, L. and Schiffer,
M.}, {\em Some inequalities for Stekloff eigenvalues},  Arch. Rat.
Mech. Anal. \textbf{57} (1974), 99-114.

%\bibitem[Iv]{Ivrii}
%  {\scshape Ivrii, V.}
%  \emph{The second term of the spectral asymptotics for a Laplace-Beltrami operator on manifolds with
%  boundary},  Funct. Anal. Appl. \textbf{14} (1980), no. 2, 98-106.

\bibitem[JM]{JM}{\scshape Jimbo, S. and Morita, Y.},
    \emph{Remarks on the behavior of certain eigenvalues on a singularly
    perturbed domain with several thin channels},
    Comm. Partial Differential Equations \textbf{17}  (1992),  no. 3-4, 523--552.

\bibitem[Jo]{JST} {\scshape Jost, J.}, {\em Partial differential
    equations}, Springer-Verlag, New-York,  2002. 


\bibitem[Kra1]{Krahn1}
  {\scshape Krahn, E.},
  \emph{\"Uber eine von Rayleigh formulierte Minimaleigenschaft des
    Kreises},
  Math. Ann. \textbf{94} (1924), 97--100.

\bibitem[Kra2]{Krahn2}
  {\scshape Krahn, E.}
  \emph{\"Uber Minimaleigenschaften der Kugel in drei und  mehr
  Dimensionen}, Acta Comm. Unic. Dorpat, A9 (1926), 1-44.

\bibitem[Kro]{Kroger}
  {\scshape Kr\"oger, P.},
  \emph{Upper bounds for the Neumann eigenvalues on a bounded domain
    in Euclidean space}, J. Funct. Anal. \textbf{106} (1992), no. 2,
  353--357.

%\bibitem[KS]{KS}{\scshape Kuttler, J. and  Sigillito, V.}, {\em  An inequality of a Stekloff eigenvalue by %the method of
%defect}, Proc. Amer. Math. Soc. \textbf{20} (1969) 357--360.

\bibitem[Le]{Levin} {\scshape Levin, B. Ya.}, \emph{Lectures on
entire functions}, in collaboration with and with a preface by Yu.
Lyubarskii, M. Sodin and V. Tkachenko. Translations of Mathematical
Monographs, 150. American Mathematical Society, Providence, RI,
1996.

%\bibitem[Me]{Melrose}
%  {\scshape Melrose, R.}
%  \emph{Weyl's conjecture for manifolds with concave boundary. Geometry of the Laplace
%  operator}, Proc. Sympos. Pure Math., XXXV (1980), 257-274.

\bibitem[Na]{Nad}
{\scshape Nadirashvili, N.}, {\em Isoperimetric inequality for the
second eigenvalue of a sphere}, J. Differential Geom. \textbf{61} (2002), no.
2, 335--340.

%\bibitem[Pa]{Payne} {\scshape Payne, L.}, {\em Isoperimetric
%inequalities and their applications}, SIAM Review, \textbf{9}, No. 3
%(1967), 453--488.

\bibitem[Pol]{Polya}
  {\scshape Polya, G.},
  \emph{Induction and analogy in mathematics. Mathematics and
    plausible reasoning, vol. I.},
  University Press, (1954), Princeton, N. J.

%\bibitem[Po2]{Polya2}
  %{\scshape Polya, G.},
  %\emph{On the eigenvalues of vibrating membranes},
  %Proc. London Math. Soc. (3) \textbf{11}, (1961), 419--433.

\bibitem[Ra]{Ra} {\scshape Rayleigh,  J.W.S.}, {\em The theory of sound}, Vol. 1, McMillan, London, 1877.

%\bibitem[RS]{RS} {\scshape Reed, M. and Simon, B.}, {\em Methods of modern mathematical physics. IV. Analysis of operators}, Academic Press, New-York-London, 1978.


%\bibitem[Sa]{Sand} {\scshape Sandgren, L.}, {\em A vibration
%problem}, Medd. Lunds Univ. Mat. Sem. 13 (1955), 1--84.


\bibitem[SY]{SY}
  {\scshape Schoen, R. and Yau, S.-T.},
  \emph{Lectures on differential geometry},
  International Press, (1994), Cambridge, MA.

%\bibitem[Sh]{Sh} {\scshape Shamma, S.}, {\em Asymptotic behavior of
   % Stekloff eigenvalues and eigenfunctions}, 
  %SIAM J. Appl. Math. \textbf{20}, No. 3 (1971), 482--490.

\bibitem[St]{St}{\scshape Stekloff, M.}, {\em Sur les probl\`emes
fondamentaux de la physique math\'ematique}, Ann. Sci. Ecole Norm.
Sup. \textbf{19} (1902), 455-490.

\bibitem[Sz]{Szego1}
  {\scshape Szeg{\"o}, G.},
  \emph{Inequalities for certain eigenvalues of a membrane of given
    area}, J. Rational Mech. Anal. \textbf{3}, (1954), 343--356.

%\bibitem[Ta]{Taylor}
%{\scshape Taylor, M.}, \emph{Partial differential equations II.
%Qualitative studies of linear equations}, Applied Mathematical
%Sciences 116, Springer-Verlag, New York, 1996.

%\bibitem[US]{US} {\scshape Uhlmann, G. and Sylvester, J.} {\em The Dirichlet to Neumann map and
%applications},  Inverse problems in partial differential equations
%(Arcata, CA, 1989), 101--139, SIAM, Philadelphia, PA, 1990.

\bibitem[Weinb]{Wein}
  {\scshape Weinberger, H. F.}
  \emph{An isoperimetric inequality for the {$N$}-dimensional free
    membrane problem}, J. Rational Mech. Anal. \textbf{5} (1956),
  633--636.

\bibitem[Weinst]{Weinst} {\scshape Weinstock, R.} {\em Inequalities for a
classical eigenvalue problem}, J. Rat. Mech. Anal. \textbf{3}
(1954), 745-753.

%\bibitem[Weyl]{Weyl}
  %\emph{Das asymptotische Verteilungsgesetz der Eigenwerte linearer partieller Differentialgleichungen},
 % Math. Ann. \textbf{71} (1911),441-479.


%\bibitem[WK]{WK} {\scshape Wolf, A. and  Keller, J.} {\em Range of the first two eigenvalues of the Laplacian},
%Proc. Roy. Soc. London, Ser. A, \textbf{447} (1994), 397--412.

\end{thebibliography}
\end{document}